\documentclass[reqno,oneside,10pt]{amsart}
\usepackage{geometry}               %
\usepackage{graphicx}
\usepackage[dvipsnames]{xcolor}
\usepackage{amssymb,amsmath,amsthm,amsfonts}
\usepackage{cancel}

\usepackage[shortlabels]{enumitem}

\usepackage{hyperref}
\usepackage[capitalise, nameinlink, noabbrev]{cleveref} 
\hypersetup{colorlinks=true, 
	linkcolor=blue,
	citecolor=red,
}
\allowdisplaybreaks

\usepackage[english]{babel}
\usepackage[utf8]{inputenc}

\usepackage{mathtools}
\usepackage{esint} 
\usepackage{dsfont}

\usepackage{comment} 

\usepackage[sc]{mathpazo}

\usepackage{tikz}
\usepackage{xfrac}
\newtheorem{theorem}{Theorem}[section]
\newtheorem{proposition}[theorem]{Proposition}

\newtheorem{lemma}[theorem]{Lemma}

\theoremstyle{definition}

\theoremstyle{remark}
\newtheorem{remark}[theorem]{Remark}
\numberwithin{equation}{section}

\newcommand{\eps}{\varepsilon}

\newcommand{\R}{{\mathbb{R}}}

\title{Fine structure of the two-phase Bernoulli free boundaries in 2D}
\author[L.~Ferreri, L.~Spolaor, B.~Velichkov]{Lorenzo Ferreri, Luca Spolaor, Bozhidar Velichkov}

\address{Lorenzo Ferreri \newline \indent
Classe di Scienze, Scuola Normale Superiore,\newline
\indent
Piazza dei Cavalieri 7, 56126 Pisa (Italy)}
 \email{lorenzo.ferreri@sns.it}

\address{Luca Spolaor \newline \indent
 	Department of Mathematics, UC San Diego, \newline \indent
 	AP\&M, La Jolla, California, 92093, USA}
 \email{lspolaor@ucsd.edu}

\address{Bozhidar Velichkov \newline \indent
Dipartimento di Matematica, Universit\`a di Pisa \newline \indent
Largo Bruno Pontecorvo, 5, 56127 Pisa, Italy}
\email{bozhidar.velichkov@unipi.it}

\begin{document}

\thanks{{\bf Acknowledgments.}
B.V. acknowledges the support of the European Research Council (ERC) via the project ERC FiRM - \it Fine structure and regularity of stationary and moving free boundaries \rm (grant agreement No. 101230705). L.S. acknowledges the support of the NSF Career Grant DMS 2044954. L.F. acknowledges the support of INdAM - GNAMPA}
\subjclass[2010] {
35R35. 
}

\begin{abstract} 
We prove that the branching set of a solution to a two-dimensional two-phase Bernoulli problem with constant coefficients is locally finite. We do this via a Weierstrass representation formula, which allows to transform the problem into a new geometric two-phase problem for capillary minimal surfaces. We also apply this method to the obstacle problem establishing a connection between the directional derivatives of solutions to the obstacle problem and the linear thin two-membrane problem.
\end{abstract}

\maketitle

\section{Introduction}
This paper is dedicated to the local structure of the free boundaries of solutions to the two-phase Bernoulli problem, that is functions $u\in C^{0,1}(B_1)\cap C^{1,\alpha}(\overline\Omega_\pm\cap B_1)$, whose supports $\Omega_\pm:=\{\pm u>0\}\cap B_1$ are disjoint sets with $C^{1,\alpha}$-regular free boundaries $\partial\Omega_\pm\cap B_1$, and
\begin{equation}\label{eq:twophase}
    \begin{cases}
        \Delta u^\pm=0&\text{in }\Omega_\pm,\\
        u^\pm =0&\text{on }\partial\Omega_\pm\cap B_1,\\
        |\nabla u^\pm|^2=\Lambda^\pm&\text{on }\partial\Omega_\pm\setminus\partial\Omega_\mp \cap B_1,\\
          |\nabla u^\pm|^2\ge\Lambda^\pm&\text{on }\partial\Omega_+\cap\partial\Omega_- \cap B_1,\\
    |\nabla u^+|^2-|\nabla u^-|^2=\Lambda^+-\Lambda^-&\text{on }\partial\Omega_+\cap\partial\Omega_- \cap B_1,
    \end{cases}
\end{equation}
for some positive real constants $\Lambda^+>0$ and $\Lambda^->0$. We remark that, at least for variational solutions, the smoothness assumption is justified by \cite{SpolaorVelichkov2019:EpiperimetricBernoulli2D} and  \cite{DePhilippisSpolaorVelichkov2021:TwoPhaseBernoulli}.\medskip

\noindent We investigate the structure of the coincidence set $\partial\Omega_u^+\cap\partial\Omega_u^-$ and in particular of the set of branching points 
\[
\mathcal B(u):=\partial\Omega_+\cap\partial\Omega_-\cap\overline{\text{\rm int}(\{u=0\})},
\]
where $\text{\rm int}(\{u=0\})$ denotes the interior part of the zero set $\{u=0\}$. We show that the set $\mathcal B(u)$ is locally finite, which concludes the analysis of the local structure of the planar two-phase free boundaries arising from \eqref{eq:twophase} completing the analysis from  \cite{AltCaffarelliFriedman1984:TwoPhaseBernoulli}, \cite{SpolaorVelichkov2019:EpiperimetricBernoulli2D,DePhilippisSpolaorVelichkov2021:TwoPhaseBernoulli}, and \cite{DePhilippisSpolaorVeluchkov2021:QuasiConformal2D} in the plane.
    \begin{theorem}\label{t:main} Let $u\in C^{0,1}(B_1)\cap C^{1,\alpha}(\overline\Omega_\pm\cap B_1)$, with $B_1\subset \R^2$, be a nontrivial solution of \eqref{eq:twophase}. Then $\mathcal B(u)$ is locally finite in $B_1$.
\end{theorem}

Combining \cref{t:main} with the results in \cite{DePhilippisSpolaorVeluchkov2021:QuasiConformal2D}, in the case $\Lambda^\pm=1$ one can obtain a more precise local description of the free-boundary at branching points. We do not know whether \cref{t:main} holds in the case of analytic weights $\Lambda^\pm \in C^\omega$, but we can easily construct a counterexample in the case $\Lambda^\pm\in C^\infty \setminus C^\omega$ (see Section \ref{ss:counter}).

\begin{theorem}\label{thm:smoothnotanalytic}
    Let $K\subset [-1,1]$ be compact. Then there exists $u$ solution to \eqref{eq:twophase}, with coefficients $\Lambda^\pm\in C^\infty(B_1)\setminus C^\omega(B_1)$ depending on $K$, such that $\mathcal B(u)=\partial K$.
\end{theorem}

Finally, using the techniques of this paper, we revisit a celebrated result by Sakai in \cite{SakaiSchwarzFunction,SakaiBranchPointsObstacle} for solutions of the two-dimensional obstacle problem, by proving the following. We restrict to the case of right hand side equal to $1$ and observe that the case of analytic right hand side is equivalent to the constant one thanks to the last section in \cite{SakaiBranchPointsObstacle}. Notice that the $C^2$ regularity assumption in the next theorem is justified by \cite{CaSh}, Section 2, (f), case 1 (see Remark \ref{rmk:cash}).

\begin{theorem}\label{t:obst}
    Let $u\in C^{1,1}(B_1)\cap C^{2}(\overline{\Omega})$, with $B_1\subset \R^2$ and $\Omega=\{u>0\}\cap B_1$ be a nonnegative solution to 
    \[
    \begin{cases}
        \Delta u = 1 &\text{in }\Omega\\
        u=0=|\nabla u| &\text{on }\partial \Omega\cap B_1\,.
    \end{cases}
    \] 
    Then the set of branching points of the free boundary of $u$, defined by
    \[
    \mathcal B(u):=\left\{x\in\partial \Omega\cap B_1\,:\,\lim_{r\to0}\frac{|B_r(x)\cap \Omega|}{|B_r(x)|}=1\right\}\cap \overline{{\rm int}(\{u=0\})}\,,
    \]
    is locally finite.
\end{theorem}

Branching (or flat) points are singular points of geometric variational problems where the tangent cone to the geometric solution is a plane (thus regular) but counted with multiplicity strictly larger than one. The main issue is to study the dimension of flat points were the density is not locally constant, that is a geometric unique continuation type problem. Such results are very scarce in the literature. The situation which is better understood is the interior regularity for area minimizing currents in any codimension, both integer and mod $p$ (see \cite{Almgren2000:AlmgrensBigRegularityPaper,Sheldon, DeLellisSpadaro2016:AreaMinimizingCurrents1LpEstimates,DeLellisSpadaro2016:AreaMinimizingCurrents2CenterManifold,DeLellisSpadaro2016:AreaMinimizingCurrents3BlowUp, DeLellisSpadaroSpolaor3:BlowUp,Dep1,Dep2}). In the context interior branching points for free boundary problems the only available results are for the two dimensional obstacle problem (see \cite{SakaiBranchPointsObstacle,SakaiSchwarzFunction}) and for the thin obstacle problem in every dimension (see \cite{Lewy,GarofaloPetrosyan2009:NewMonotonicityThinObstacle,FocardiSpadaro2018:MeasureFreeBdThinObstacle}). The major difference between these two situations is that for area minimizing currents the problem is local, while for free boundary problems the problem is non-local: that is the free boundaries do not solve an equation themselves. In this respect, a more similar situation is the investigation of boundary branching points, see for instance \cite{FerreriSpolaorVelichkov2024:BoundaryBranchingOnePhaseBernoulli,DeLellisDePhilippisHirschMassaccesi2023,DLboundary,FlescherResende2024}.

In this respect, the closest known result to \cref{t:main} is that of Sakai in \cite{SakaiSchwarzFunction,SakaiBranchPointsObstacle}, and in fact, although fundamentally different, both our techniques and Sakai's make heavy use of complex analysis tools and thus are fundamentally two-dimensional. 

Our main idea is inspired by the work \cite{HauswirthHeleinPacard} and \cite{Traizet}, where the authors draw a beautiful parallel between solutions of the one-phase Bernoulli problem and free boundary minimal surfaces in the upper half-space. In Section \ref{ss:gp} we revisit this idea to show that \eqref{eq:twophase} is connected to a \emph{new two-phase capillary problem} by constructing compatible Weierstrass parametrization from the solution of the two-phase Bernoulli problem. 
In Section \ref{s:nonlinear-two-membrane} we transform the geometric problem into a nonlinear thin two-membrane problem for the area functional by simply choosing suitable graphical coordinates. Finally in Section \ref{s:conclusion} we leverage the main result of \cite{FerreriSpolaorVelichkov2024:FineStructureThinTwoMembrane2D} to reach the conclusion.   

We hope that the techniques introduced in this paper could be useful in the investigation of similar issues for capillary problems and boundary regularity issues for area minimizing currents.

\section{A two-phase capillary problem}\label{ss:gp}

In this section we observe that in dimension $2$ solutions to the two-phase problem induce minimal surfaces that solve a two-phase capillary problem.

\begin{theorem}[Two-phase capillary problem]\label{thm:main-weierstrass}
Given $\Lambda_\pm>0$, there is $\eps>0$ such that the following holds. 
Suppose that $u:B_2\to\R$ is a solution of \eqref{eq:twophase} in $B_2$ which is $\eps$-flat in the sense that 
\[
|u(x,y)-(\Lambda_+^{\sfrac12}y_+-\Lambda_-^{\sfrac12}y_-)|\le \eps\quad\text{for all}\quad (x,y)\in B_2.
\]
There exist functions $\psi_1^\pm,\psi_2^\pm \colon \overline\Omega_\pm\cap B_{1}\to \R$ such that the following properties hold. Define  
\[
\Omega_\pm:=\{\pm u>0\}\ ,\quad v^+:=\Lambda_{-}^{-\sfrac12}u^+\ ,\quad v^-:=\Lambda_{+}^{-\sfrac12}u^-\ ,\quad \lambda:=\left(\Lambda_+/\Lambda_-\right)^{\sfrac12}\,,
\]
and let 
\[
\Psi^\pm=(\psi_1^\pm,\psi_2^\pm, \lambda^{\mp1}v^\pm)\colon \overline{\Omega_{\pm}}\to \R^3\cap\{\pm x_3\geq 0\}
\]
and $M^\pm=\Psi^\pm(\overline \Omega_\pm)$. Set $\Gamma_\pm=B_2\cap \partial  \Omega_\pm$ and $\partial M^\pm=\Psi^\pm(\Gamma_\pm)$. 
\begin{enumerate}
    \item $\partial M^\pm \subset \{x_3=0\}$, and moreover the following boundary conditions hold:\\ on the two-phase free boundary we have
    \begin{equation}\label{eq:TPcoincidence}
    \Psi^+(x,y)=\Psi^-(x,y)\in \partial M^+\cap \partial M^-\qquad\text{for all}\qquad (x,y)\in \Gamma_+\cap \Gamma_-,
    \end{equation}
    while on the one-phase part
    we have 
    \begin{align}\label{eq:OPspeed}
    d(\psi_1^\pm+i\psi_2^\pm)=\frac{(1 + \lambda)}{2\lambda} d\bar{z}\qquad\text{on}\qquad ( \Gamma_+\setminus\Gamma_-)\cup (\Gamma_-\setminus\Gamma_+)\,.
 \end{align} 
    \item The following capillary type problem holds
    \[
    \begin{cases}
        H_{M^\pm}=0& \text{on }\  M^{\pm}\cap \{\pm x_3>0\}\\
        \nu^\pm\cdot e_3= \pm \frac{1-\lambda^2}{1+\lambda^2}& \text{on }\ \partial M^\pm\setminus \partial M^\mp\\
          \nu^\pm\cdot e_3\le \pm \frac{1-\lambda^2}{1+\lambda^2}& \text{on }\ \partial M^+\cap \partial M^-\\
        {\lambda}\left(1 + e_3\cdot\nu^+\right)=\frac{1}{\lambda}\left(1 + e_3\cdot\nu^-\right)& \text{on }\ \partial M^+\cap \partial M^- 
    \end{cases}
    \]
    where $H_M$ denotes the mean curvature of $M$.
\end{enumerate}
\end{theorem}

\begin{proof}
We proceed step-by-step using the Weierstrass representation with the data $g^\pm = 2\partial_z v^\pm$ and $f^+ = 1/\lambda$, $f^- = \lambda$. Let $z = x + iy \in B_2 \subset \mathbb{C}$.

\noindent\textbf{Step 1: Setting up the Weierstrass Data}
Since $u$ is harmonic in $\Omega_\pm = \{\pm u > 0\}$, the scaled functions $v^\pm$ are also harmonic in their respective domains. We define the holomorphic functions $g^\pm$ as the complex gradients of $v^\pm$:
\[ 
g^\pm = 2\partial_z v^\pm = v_x^\pm - i v_y^\pm 
\]
Notice that $|g^\pm| = |\nabla v^\pm|$. We define the analytic functions $f^\pm$ as constants:
\[ 
f^+ = \frac{1}{\lambda}, \quad f^- = \lambda\,.
\]

\noindent\textbf{Step 2. Constructing the Minimal Surfaces $M^\pm$.}
We construct the surfaces $M^\pm$ via the parameterizations $\Psi^\pm: \overline{\Omega}_\pm \to \mathbb{R}^3$, using the standard Weierstrass-Enneper formulas:
\[ 
\Psi^\pm(z) = \text{Re} \int_{z_0}^z \left( \frac{1}{2}f^\pm(1-(g^\pm)^2), \frac{i}{2}f^\pm(1+(g^\pm)^2), f^\pm g^\pm \right) d\zeta .
\]
We notice that the maps $\Psi^\pm$ are well defined for all $z\in \overline \Omega_\pm$ since, by the flatness assumption, $\overline{\Omega}_\pm$ are simply connected domains.
Because $M^\pm:=\{\Psi^\pm(z)\,:\, z\in \overline{\Omega}_\pm\}$ are defined via Weierstrass data and $g^\pm$ are holomorphic in $\Omega_\pm$, they are automatically minimal surfaces in the interior, which immediately proves $H_{M^\pm} = 0$ on $\Phi^+(\Omega_+)\cup\Phi^-(\Omega_-)$.

For the third coordinate, $x_3^\pm$, we have:
\[ 
x_3^\pm = \text{Re} \int f^\pm g^\pm dz = \text{Re} \int f^\pm (v_x^\pm - i v_y^\pm) (dx + i dy) = f^\pm v^\pm. 
\]
Since $v^\pm = 0$ on the free boundary $\Gamma_+ \cup \Gamma_-$, the third coordinate vanishes there. This proves the inclusion $\partial M^\pm \subset \{x_3=0\}$.\medskip

\noindent\textbf{Step 3. Transformation of the free boundaries.} 
 Let $s^\pm$ be the arc length parameter of $\Gamma^\pm$, and let $\tau^\pm = \tau_1^\pm + i\tau_2^\pm$ be its unit tangent vector. We choose the orientation of $\Gamma_+$ and $\Gamma_-$ in such a way that $\nu^+:=i\tau^+$ is the normal pointing inwards $\Omega_+$, while $\nu^-:=i\tau^-$ is the normal pointing outwards $\Omega_-$.

Since $v^\pm = 0$ on $\Gamma_\pm$, the gradients $\nabla v^\pm:=v^\pm_x+iv^\pm_y$ are normal to $\Gamma_\pm$. Thus, $\nabla v^+ = i|\nabla v^+|\tau^+$ and $\nabla v^- = i|\nabla v^-|\tau^-$. In complex notation:
\begin{align*}
    g^\pm &= \overline{\nabla v^\pm} = -i|\nabla v^\pm|\overline{\tau^\pm} \implies (g^\pm)^2 = -|\nabla v^\pm|^2\overline{\tau^\pm}^2
\end{align*}

We shall denote $\Psi^\pm_{1,2}:=\psi_1^\pm+i\psi_2^\pm$, that is 
\[
\Psi^\pm_{1,2} =\overline{\int_{z_0}^z \frac{f^\pm}2}-\int_{z_0}^z \frac{1}2f^\pm( g^\pm)^2
\]
Evaluating the differential of the $\Psi_{1,2}^\pm$ along $\Gamma_\pm$ using the identities 
\[
d\Psi^\pm_{1,2} = \frac{1}{2}\overline{f^\pm} d\bar{z} - \frac{1}{2}f^\pm (g^\pm)^2 dz\qquad\text{and}\qquad dz=\tau^\pm ds^\pm
\] 
along $\Gamma$ we get:
\begin{align*}
    d\Psi^+_{1,2} &= \frac{1}{2\lambda}\overline{\tau^+} ds^+ - \frac{1}{2\lambda}(-|\nabla v^+|^2\overline{\tau^+}^2)\tau^+ ds^+ = \frac{\overline{\tau^+}}{2\lambda}(1 + |\nabla v^+|^2) ds^+\,, \\
    d\Psi^-_{1,2} &= \frac{\lambda}{2}\overline{\tau^-} ds^- - \frac{\lambda}{2}(-|\nabla v^-|^2\overline{\tau^-}^2)\tau^- ds^- = \frac{\overline{\tau^-}}{2}\lambda(1 + |\nabla v^-|^2) ds^-\,.
\end{align*}

For \eqref{eq:TPcoincidence} to hold, the tangent parts of the form $d\Psi^\pm_{1,2}$ must be identical on $\Gamma^+\cap \Gamma^-$ and must have the same integral on $\Gamma^+\setminus \Gamma^-$ and $\Gamma^-\setminus \Gamma^+$ in between any two branchings. On $\Gamma^+\cap \Gamma^-$ the differentials yields:
\[ 
\frac{1}{2\lambda}(1 + |\nabla v^+|^2) = \frac{\lambda}{2}(1 + |\nabla v^-|^2) \iff |\nabla v^+|^2 - \lambda^2 |\nabla v^-|^2 = \lambda^2 - 1. 
\]
Substituting $v^+ = \Lambda_-^{-\sfrac{1}{2}}u^+$ and $v^- = \Lambda_+^{-\sfrac{1}{2}}u^-$ back into the equation:
\[ 
\Lambda_-^{-1}|\nabla u^+|^2 - \lambda^2 \Lambda_+^{-1}|\nabla u^-|^2 = \lambda^2 - 1. 
\]
Multiplying by $\Lambda_-$ and using $\lambda^2 = \Lambda_+ / \Lambda_-$ gives the exact two-phase free boundary condition:
\[ 
|\nabla u^+|^2 - |\nabla u^-|^2 = \Lambda_+ - \Lambda_-\,. 
\]

On the remaining region we have, using the same computations as above and the one-phase free-boundary conditions $|\nabla v^\pm|^2=\lambda^{\pm2}$:
\[
    d\Psi^\pm_{1,2} = \frac{(1 + \lambda^2)}{2\lambda} \bar{\tau}^\pm\,ds^\pm=\frac{(1 + \lambda)}{2\lambda} d\bar{z}\,,
 \]
 which concludes the proof of \eqref{eq:OPspeed} and, by integration between two branchings also of \eqref{eq:TPcoincidence}.

 \medskip

\noindent\textbf{Step 4: The Capillary Transmission Condition.}
By choosing the downward-pointing normal, the vertical component is given by $e_3\cdot\nu^\pm = \frac{1 - |g^\pm|^2}{1 + |g^\pm|^2}$. Since $|\nabla v^\pm|^2\geq \lambda^{\pm 2}$, this yields 
$$1 + e_3 \cdot \nu^\pm = \frac{2}{1 + |g^\pm|^2}\leq \frac{2}{1+\lambda^{\pm2}}\,,$$
that is 
\[
e_3 \cdot \nu^\pm\leq \pm \frac{1-\lambda^2}{1+\lambda^2}\quad\text{on}\quad\Gamma^\pm\,.
\]
We verify the matching condition on $\partial M^+ \cap \partial M^-$:
\[ 
\lambda \left( \frac{2}{1 + |g^+|^2} \right) = \frac{1}{\lambda} \left( \frac{2}{1 + |g^-|^2} \right) \implies \frac{1}{\lambda}(1 + |g^+|^2) = \lambda(1 + |g^-|^2).
\]
Since $|g^\pm| = |\nabla v^\pm|$, this is algebraically identical to the metric matching condition we proved in Step 3. Thus, the free boundary condition guarantees the capillary transmission holds.\medskip

\noindent\textbf{Step 5: The Capillary One-Phase Condition.} 
We finally determine the free boundary condition on the one-phase parts of the geometric  free boundary $\partial M^\pm \setminus \partial M^\mp$ as a consequence of the free boundary condition that $u$ satisfies of $\partial\Omega_\pm\setminus\partial\Omega_\mp$. 

On the positive one-phase boundary $\partial \Omega_+ \setminus \partial \Omega_-$, we have $|\nabla u^+| = \Lambda_+^{1/2}$. Recalling our scaling $v^+ = \Lambda_-^{-1/2} u^+$, the gradient magnitude on this boundary becomes:
\[
|g^+| = |\nabla v^+| = \Lambda_-^{-1/2}|\nabla u^+| = \Lambda_-^{-1/2}\Lambda_+^{1/2} = \left(\frac{\Lambda_+}{\Lambda_-}\right)^{1/2} = \lambda\,.
\]
Using the downward-pointing normal formula 
$$e_3\cdot\nu^\pm = \frac{1 - |g^\pm|^2}{1 + |g^\pm|^2},$$ 
we substitute $|g^+| = \lambda$ to find the vertical component of the normal on $\partial M^+ \setminus \partial M^-$:
\[
\nu^+ \cdot e_3 = \frac{1 - \lambda^2}{1 + \lambda^2}\,.
\]

Similarly, on the negative one-phase boundary $\partial \Omega_- \setminus \partial \Omega_+$, the condition is $|\nabla u^-| = \Lambda_-^{1/2}$. Using the scaling $v^- = \Lambda_+^{-1/2} u^-$, the gradient magnitude is:
\[
|g^-| = |\nabla v^-| = \Lambda_+^{-1/2}|\nabla u^-| = \Lambda_+^{-1/2}\Lambda_-^{1/2} = \left(\frac{\Lambda_-}{\Lambda_+}\right)^{1/2} = \frac{1}{\lambda}\,.
\]
Substituting $|g^-| = 1/\lambda$ into the normal formula yields:
\[
\nu^- \cdot e_3 = \frac{1 - (1/\lambda)^2}{1 + (1/\lambda)^2} = \frac{\lambda^2 - 1}{\lambda^2 + 1} = -\frac{1 - \lambda^2}{1 + \lambda^2}.
\]
\end{proof}

\begin{remark}\label{remark:forms}
In the proof of \cref{thm:main-weierstrass} an alternative way to find the functions $\psi_1^\pm$ and $\psi_2^\pm$ is by integrating the closed differential forms
$\alpha_1^\pm:B_1\cap\overline\Omega_\pm\to\R^2$ and $\alpha_2^\pm:B_1\cap\overline\Omega_\pm\to\R^2$ 
defined as
\begin{align*}
    \begin{split}
        {\alpha}_1^{+}& =  \frac1{\lambda}\left(\frac{1}{2}\left(1 - \left( v^{+}_x \right)^2+\left( v^{+}_y \right)^2\right)\,dx-v^{+}_x v^{+}_y\,dy\right), \\
        {\alpha}_1^{-}& =  {\lambda}\left(\frac{1}{2}\left(1 - \left( v^{-}_x \right)^2+\left( v^{-}_y \right)^2\right)\,dx-v^{-}_x v^{-}_y\,dy\right),\\      {\alpha}_2^{+} &  = \frac{1}{\lambda}\left(-v^{+}_x v^{+}_y\,dx+\frac{1}{2}\left(1 + \left( v^{+}_x \right)^2-( v^{+}_y )^2\right)\,dy\right),\\
         {\alpha}_2^{-} &  = \lambda\left(-v^{-}_x v^{-}_y\,dx+\frac{1}{2}\left(1 + \left( v^{-}_x \right)^2-( v^{-}_y )^2\right)\,dy\right).
    \end{split}
\end{align*}
We notice that these are exactly the real parts of the first two components in the Weierstrass representation formula. 
We will exploit this formulation in the appendix in the context of the obstacle problem.
\end{remark}

\section{Nonlinear thin two-membrane problem}\label{s:nonlinear-two-membrane}

We are now going to reparametrize the geometric problem above as a thin-two membrane problem for the area functional with capillary condition on the $x_1x_3$-plane. We start by inverting the $x_1x_3$ coordinates of the maps $\Psi^\pm$.

\begin{lemma}
Suppose $0 \in \partial\Omega_+ \cap \partial\Omega_-$ is a branching point. Let the transformation $T^\pm$ be defined as $T^\pm(x,y) = (\psi_1^\pm(x,y), f^\pm v^\pm(x,y))$, with the notation from \cref{thm:main-weierstrass}. Then, $T^\pm$ is a $C^1$ diffeomorphism onto its image.
\end{lemma}

\begin{proof}
We prove this by showing the transformation has a non-vanishing Jacobian determinant. By definition, the map is $T^\pm(x,y) = (s, t) = (\psi_1^\pm(x,y), f^\pm v^\pm(x,y))$.

From the Weierstrass representation, the differential of the horizontal components is:
\[
d(\psi_1^\pm + i\psi_2^\pm) = \frac{1}{2}f^\pm d\bar{z} - \frac{1}{2}f^\pm (g^\pm)^2 dz
\]
Substituting $g^\pm = v_x^\pm - i v_y^\pm$, $dz = dx + i dy$, and $d\bar{z} = dx - i dy$, we can isolate the real part to find the gradients of $\psi_1^\pm$:
\begin{align*}
\partial_x \psi_1^\pm &= \frac{1}{2} f^\pm (1 - (v_x^\pm)^2 + (v_y^\pm)^2) \\
\partial_y \psi_1^\pm &= -f^\pm v_x^\pm v_y^\pm
\end{align*}

The Jacobian matrix of $T^\pm$ is given by:
\[
DT^\pm = \begin{pmatrix} 
\partial_x \psi_1^\pm & \partial_y \psi_1^\pm \\ 
\partial_x (f^\pm v^\pm) & \partial_y (f^\pm v^\pm) 
\end{pmatrix} 
= \begin{pmatrix} 
\frac{1}{2} f^\pm (1 - (v_x^\pm)^2 + (v_y^\pm)^2) & -f^\pm v_x^\pm v_y^\pm \\ 
f^\pm v_x^\pm & f^\pm v_y^\pm 
\end{pmatrix}
\]
Calculating the determinant $J^\pm = \det(DT^\pm)$:
\begin{align*}
J^\pm &= \frac{1}{2} (f^\pm)^2 v_y^\pm (1 - (v_x^\pm)^2 + (v_y^\pm)^2) - (-f^\pm v_x^\pm v_y^\pm)(f^\pm v_x^\pm) \\
&= \frac{1}{2} (f^\pm)^2 v_y^\pm [1 - (v_x^\pm)^2 + (v_y^\pm)^2 + 2(v_x^\pm)^2] \\
&= \frac{1}{2} (f^\pm)^2 v_y^\pm (1 + |\nabla v^\pm|^2)\,.
\end{align*}

Since $u^\pm\in C^{1,\alpha}(\overline \Omega_\pm)$ and $|\nabla u^\pm|^2(0)=\Lambda^\pm>0$, the Jacobian is nonnegative in a neighborhood of $0$ and we are done.

\end{proof}

Next we derive the desired thin-two membrane problem

\begin{proposition}
Let $0 \in \partial\Omega_+ \cap \partial\Omega_-$ be a branching point, and let $T^\pm$ be the $C^1$ diffeomorphisms defined by $T^\pm(x,y) = (\psi_1^\pm(x,y), f^\pm v^\pm(x,y))$. Define the functions $w^\pm(s, t) = \psi_2^\pm \circ (T^\pm)^{-1}(s,t)$. Then $w^\pm$ are a solution of the following nonlinear thin two-membrane problem:
\[
\begin{cases}
    \nabla \cdot \left( \frac{\nabla w^\pm}{\sqrt{1 + |\nabla w^\pm|^2}} \right) = 0 & \text{in } \{t \neq 0\}, \\
    w^+ \leq w^- & \text{on } \{t = 0\}, \\
    \frac{-\partial_t w^\pm}{\sqrt{1 + |\nabla w^\pm|^2}} = \pm \frac{1-\lambda^2}{1+\lambda^2} & \text{on } \{t=0\} \cap \{w^+ > w^-\}\\
    \frac{-\partial_t w^\pm}{\sqrt{1 + |\nabla w^\pm|^2}} \leq \pm \frac{1-\lambda^2}{1+\lambda^2}& \text{on } \{t=0\} \cap \{w^+ = w^-\}
\end{cases}
\]
\end{proposition}

\begin{proof} The proof follows trivially from the previous Proposition, with the exception of the ordering condition. Thanks to \cite{SpolaorVelichkov2019:EpiperimetricBernoulli2D,DePhilippisSpolaorVelichkov2021:TwoPhaseBernoulli}, we can assume that up to a rotation of the coordinate system, in a neighborhood $(-r,r)^2$ of a two-phase point $0\in\partial\Omega_u^+\cap\partial\Omega_u^-$,  
the sets $\Omega_u^\pm$ are given by 
\begin{equation}\label{e:definition-eta+-}
\Omega_u^+=\Big\{(x,y)\ :\ y> \eta_+(x)\Big\}\quad\text{and}\quad\Omega_u^-=\Big\{(x,y)\ :\ y< \eta_-(x)\Big\},
\end{equation}
where $\eta_\pm\in C^{1,\alpha}((-r,r))$, $\eta_+(0)=\eta_-(0)=0$, $\eta_+'(0)=\eta_-'(0)=0$ and 
\[
-r<\eta_-\le\eta_+<r\quad\text{in}\quad(-r,r).
\]
Next recall that  on their respective free boundaries, the differentials of $\psi_1^\pm$ and $\psi_2^\pm$ are by \eqref{eq:OPspeed}:
\[
d\psi_1^+ = \frac{1 + \lambda^2}{2\lambda} dx \quad \text{and} \quad d\psi_1^- = \frac{1 + \lambda^2}{2\lambda} dx,
\]
\[
d\psi_2^+ = -\frac{1 + \lambda^2}{2\lambda} dy \quad \text{and} \quad d\psi_2^- = -\frac{1 + \lambda^2}{2\lambda} dy.
\]
Subtracting these equations yields
\[
w^+(s, 0) - w^-(s, 0) = -\frac{1 + \lambda^2}{2\lambda} \Big( \eta_+\left(\frac{1 + \lambda^2}{2\lambda}x\right) - \eta_-\left(\frac{1 + \lambda^2}{2\lambda}x\right) \Big)\leq0\,.
\] 
\end{proof}

\section{Conclusion of the proof}\label{s:conclusion}

\begin{proof}[Proof of \cref{t:main}]
Let $w^-$ be defined on $t \le 0$ and $w^+$ be defined on $t \ge 0$. We reflect $w^-$ across the line $\{t = 0\}$ by defining the even reflection:
$$\widetilde{w}^-(s, t) := w^-(s, -t) \quad \text{for} \quad t \ge 0.$$
Both original functions $w^\pm$ solve the interior minimal surface equation in their respective domains, which we can write in divergence form as:
$$  {\rm div}  (A(\nabla w^\pm)) = 0 \quad \text{where} \quad A(p) = \frac{p}{\sqrt{1+|p|^2}}$$
 Since the minimal surface operator only contains second derivatives and first derivatives squared, it is invariant under the reflection $t \mapsto -t$. Therefore, $\widetilde{w}^-$ also satisfies the minimal surface equation in the upper half-plane:
$$  {\rm div}  (A(\nabla \widetilde{w}^-)) = 0 \quad \text{for } t > 0\,.$$

    We define  $d(s,t) = \widetilde{w}^-(s,t) - w^+(s,t)$. 
By subtracting their equations and applying the fundamental theorem of calculus along the line segment between their gradients, we get:
$$  {\rm div}  \left( \left[ \int_0^1 DA\big(\tau \nabla \widetilde{w}^- + (1-\tau)\nabla w^+\big) \, d\tau \right] \nabla d \right) =  {\rm div}  (A(\nabla \widetilde{w}^-)) -   {\rm div}  (A(\nabla w^+))= 0.$$

Let us call this integral matrix $B(s,t)$. Because $A(p)$ is strictly monotone and smooth, $B(s,t)$ is a uniformly positive-definite matrix for bounded gradients, and we obtain a linear elliptic operator $L$:
$$  {\rm div}  (B(s,t) \nabla d) = 0 \quad \text{in } t > 0\,.$$
On the thin boundary $t=0$ we clearly have 
$$d(s,0) \ge 0\,.$$
Finally let us define $$\partial_{\nu_B} d =e_2 \cdot (B\nabla d)= -\big( A_t(\nabla \widetilde{w}^-) - A_t(\nabla w^+) \big)$$
where $A_t(\nabla w) = \frac{\partial_t w}{\sqrt{1+|\nabla w|^2}}$. This is the natural Neumann condition associated to $B$. Recall our previously derived transmission conditions on the non-coincidence set:
$$\frac{-\partial_t w^+}{\sqrt{1+|\nabla w^+|^2}} = \frac{1-\lambda^2}{1+\lambda^2} \implies A_t(\nabla w^+) = -\frac{1-\lambda^2}{1+\lambda^2}$$
$$\frac{-\partial_t w^-}{\sqrt{1+|\nabla w^-|^2}} = -\frac{1-\lambda^2}{1+\lambda^2}$$

Now we calculate the flux for the reflected function $\widetilde{w}^-$. Because $\widetilde{w}^-(s,t) = w^-(s,-t)$, the chain rule dictates that on the boundary $t=0$, the derivative flips: $\partial_t \widetilde{w}^- = -\partial_t w^-$. Therefore:
$$A_t(\nabla \widetilde{w}^-) = \frac{\partial_t \widetilde{w}^-}{\sqrt{1+|\nabla \widetilde{w}^-|^2}} = \frac{-\partial_t w^-}{\sqrt{1+|\nabla w^-|^2}} = -\frac{1-\lambda^2}{1+\lambda^2}\,,$$
which implies $\partial_{\nu_B} d=0$ where $d>0$. An analogous computations shows that $\partial_{\nu_B} d\geq0$ where $d=0$.

In conclusion w have that $d$ is a solution to
\[
\begin{cases}
    {\rm div} (B \nabla d) = 0 & \text{in } \{t > 0\}\\
    d\geq 0 & \text{on } \{t = 0\}\\
    \nu \cdot (B\nabla d)=0& \text{on } \{t = 0\}\cap \{d>0\}\\
    \nu \cdot (B\nabla d)\geq 0& \text{on } \{t = 0\}\cap \{d=0\}
\end{cases}
\]
and so in particular it is a variational solution to a variable coefficients thin obstacle, that is
\begin{equation}\label{e:variational-inequality-thin}
\int_{B_1^+} B \nabla d \cdot \nabla (d - v) \le 0
\end{equation}
for all functions
\[
v \in H^1(B_1^+), \quad v \ge 0 \text{ on } B_1' \quad \text{and} \quad v = d \text{ on } \partial B_1 \cap \{ t \ge 0 \}\,,
\]
where the matrix $B$ is uniformly elliptic with $C^{0,\alpha}$ coefficients. By \cite[Theorem 1.1]{FerreriSpolaorVelichkov2024:FineStructureThinTwoMembrane2D}, there exists a quasi-conformal homeomorphism $f:\R^2 \to \R^2$ such that the function
\[
h \coloneqq d \circ f
\]
is a solution of the harmonic thin-obstacle problem in a neighborhood of the origin, that is a solution of \eqref{e:variational-inequality-thin} with $B=I$. In particular in the ball neighborhood, the non-contact set of $d$ 
\[
\mathcal C:=\{(s,0)\in B_1'\ :\ d(s,0)>0\}
\]
is composed of a finite number of disjoint open segments, that is branching points are locally finite.
\end{proof}

\section{Counterexample in the case $\Lambda^\pm\in C^\infty\setminus C^\omega$}\label{ss:counter}

Let $u(x,y):=y^+$ and let $f\in C^\infty(\R)\setminus C^\omega(\R)$ be a non positive function. Assume that $f(0)=f^{(k)}(0)=0$, for every $k\geq 1$. Everything in the following is true up to choosing a sufficiently small neighborhood of the origin. Let $E_f=\{(x,y)\in \R^2\,:\, y> f(x)\}$ and consider the map $\Psi_f\colon \overline E_f \to \{y\geq 0\}$ defined by
\[
\Psi(x,y):=(\bar{v}^+(x,y),v^+(x,y))\,,
\]
where $v^+$ is a nonnegative harmonic function in $E_f$ with $v=0$ on $\partial E_f$, and $\bar{v}^+$ is the harmonic conjugate of $v^+$ with $\bar v^+(0,0)=0$. By construction $\Psi_f$ is a conformal map in $E_f$ and $\Psi_f \in C^\infty(\overline{E_f})$. By the Hopf maximum principle, $|\nabla v^+|>0$ on $\partial E_f$, so that we can consider the function $w^+(s,t):=u(\Psi_f^{-1}(s,t))$. Since $f\le 0$, we have that the positivity set $\Omega_+:=\{w^+>0\}=\Psi_f(\{u>0\})\subset \{t>0\}$ is contained in the upper half-plane. It is trivial to verify that $w^+\in C^{\infty}(\overline{\Omega_+})$ and moreover
\[
\begin{cases}
\begin{array}{rl}
    \Delta w^+=0&\text{in }\Omega_+\\
    w=0&\text{on }\partial \Omega_+\\
    |\nabla w|^2=\Lambda^+&\text{on }\partial \Omega_+\,.
\end{array}    
\end{cases}
\]
where $\Lambda^+(s,t):=|\nabla v^+|^{-2}(\Psi_f^{-1})\in C^{\infty}(\{t\geq 0\})$. 

Now repeating the same construction for $-f$ and the set $E_{-f}=\{(x,y)\in \R^2\,:\, y< -f(x)\}$ we obtain a function $w^-$ with positivity set $\Omega_{-}\subset \{y<0\}$ satisfying
\[
\begin{cases}
\begin{array}{rl}
    \Delta w^-=0&\text{in }\Omega_{-}\\
    w^-=0&\text{on }\partial \Omega_{-}\\
    |\nabla w^-|^2=\Lambda^-&\text{on }\partial \Omega_{-}\,.
\end{array}    
\end{cases}
\]
where $\Lambda^-(s,t):=|\nabla v^-|^{-2}(\Psi_{-f}^{-1})(s,t)\in C^{\infty}(\{t<0\})$, with $v^-(x,y)=-v(x,-y)$.
Now observe that by construction $\Lambda^+(s,0)=\Lambda^-(s,0)$. Extending both of them locally to $C^\infty$ functions, we obtain a solution $w=w^++w^-$ to the two phase Bernoulli problem with weights $\Lambda^\pm$. Since for $f\in C^{\infty}\setminus C^\omega$ the set $\{f=0\}$ can be chosen at will, we are done. 

We also notice that by construction, the solution is odd: $w^-(s,y)=-w^+(x,-t)$.

\section{Proof of \cref{t:obst}}

\subsection{Weierstrass formula and the obstacle problem} Let $u\in H^{1}(B_1)$ be a solution to the obstacle problem $\Delta u=\chi_{\{u>0\}}$ in $B_1$ and let $V:=\nabla u$. Since $u\in C^{1,1}(B_1)$, we have that the vector field $V$ is $C^{0,1}$ in $B_1$, $C^\infty$ in the positivity set $\Omega:=\{u>0\}\cap B_1$ and vanishes identically on $B_1\setminus \Omega$. We set $u_x:=\partial_xu$, $u_y:=\partial_yu$, $V_x:=\partial_xV$ and $V_y:=\partial_yV$ and consider the following vectorial counterpart of the Weierstrass differential forms from Remark \ref{remark:forms}:
\begin{align*}
    \begin{split}
        \alpha_1 & \coloneqq \frac{1}{2}\left(1- |V_x|^2 + |V_y|^2 \right)\,dx - (V_x\cdot V_y) \,dy,\\
        \alpha_2 & \coloneqq (V_x\cdot V_y) \,dx - \frac{1}{2} \left(1+ |V_x|^2 - |V_y|^2 \right)\,dy.
    \end{split}
\end{align*}
Using the equation $\Delta u=1$ in $\Omega$, one can easily check that in $\Omega$, $\alpha_1$ and $\alpha_2$ can be written as
\begin{align*}
        \alpha_1  = u_{yy}\,dx - u_{xy} \,dy\qquad\text{and}\qquad
        \alpha_2  = u_{xy} \,dx - u_{xx}\,dy.
\end{align*}
These forms are clearly closed on $\Omega$, but they also turn out to be exact on $\Omega$, without any topological assumptions on $\Omega$. Indeed, it is immediate to check that
$$\alpha_1=d(x-u_x)\quad\text{and}\quad \alpha_2=d(u_y-y)\quad\text{in}\quad\Omega.$$
In particular, the first two components of the Weierstrass transformation (\cref{thm:main-weierstrass}) for the obstacle problem, are simply given by 
$(\psi_1,\psi_2)=(x-u_x,u_y-y)$.
Moreover, the potentials $\psi_1:=x-u_x$ and $\psi_2:=u_y-y$ of $\alpha_1$ and $\alpha_2$ are precisely the conjugate harmonics of $u_x$ and $u_y$. Indeed, a straightforward computation gives the following:
\begin{lemma}\label{l:obstacle-S-and-T}
Let $u\in C^{1,1}(D)$ be a solution to the obstacle problem $\Delta u=\chi_{\{u>0\}}$ in an open set $D\subset\R^2$ and let $\Omega:=\{u>0\}\cap D$. Then, the functions $T:\Omega\to\mathbb{C}$ and $S:\Omega\to\mathbb{C}$ given by 
$$T:=(x-u_x)+iu_y\qquad\text{and}\qquad S:=(u_y-y)+iu_x,$$
 are holomorphic functions on $\Omega$ satisfying the relation $T=iS+z$.
 \end{lemma}

\subsection{On the $C^{2}$ regularity assumption} In this section we discuss the role of the $C^{2}$ regularity assumption in \cref{t:obst}, which allows to write the obstacle problem as a two-phase problem for one of the partial derivatives. 
\begin{lemma}\label{l:obstacle-smooth-parametrization}
Let $u\in C^{1,1}(D)$ be a solution to the obstacle problem $\Delta u=\chi_\Omega$ in the open set $D\subset\R^d$, where $\Omega:=\{u>0\}$. Suppose that $u\in C^{2}(D\cap\overline\Omega)$ and that $0\in \partial\Omega\cap D$ is a singular point in the first stratum, that is $u(x)=\frac12x_d^2+o(|x|^2)$ for $x\in D$. Then, there are $C^{1}$ functions $\eta_\pm:\R^{d-1}\to\R$, with $\eta_+\ge \eta_-$, such that the following holds in a neighborhood of zero: 
\begin{equation}\label{e:obstacle-existence-of-curves}
\partial\{\pm\partial_du>0\}=\big\{(x',\eta_\pm(x'))\ :\ x'\in\R^{d-1}\big\},
\end{equation}
$$\text{\rm Reg}(\partial\Omega)=\big\{(x',\eta_+(x')),(x',\eta_-(x'))\ :\ x'\in\R^{d-1},\ \eta_-(x')<\eta_+(x')\big\},$$
$$\text{\rm Sing}(\partial\Omega)=\big\{(x',\eta_+(x'))\ :\ x'\in\R^{d-1},\ \eta_-(x')=\eta_+(x')\big\},$$
$$\text{\rm int}(\{u=0\})=\big\{(x',x_d)\ :\ x'\in\R^{d-1},\ \eta_-(x')<x_d<\eta_+(x')\big\},$$
where we use the usual notation
$$\text{\rm Reg}(\partial\Omega)=\Big\{x\in D\cap\partial\Omega\ :\ \text{the Lebesgue density of $\Omega$ at $x$ is $1/2$}\Big\};$$
$$\text{\rm Sing}(\partial\Omega)=\Big\{x\in D\cap\partial\Omega\ :\ \text{the Lebesgue density of $\Omega$ at $x$ is $1$}\Big\}.$$
Vice versa, if there are $C^{1}$ curves $\eta_+\ge \eta_-$ such that \eqref{e:obstacle-existence-of-curves} holds, then $u\in C^2(D\cap\overline\Omega)$.
\end{lemma}
\begin{proof}
Set $C_{r,\ell}:=B_r'\times(-\ell,\ell)$. Since $\{\partial_du\neq 0\}\subset\Omega$, we know that $\partial_{dd}u>0$ in $C_{r,\ell}\cap\{\partial_du\neq 0\}$ for some small $r>0$. We can choose $\ell$ and then $r$ small enough such that $\partial_du(x',\ell)>0$ and $\partial_du(x',\ell)<0$ for all $x'\in B_r'$. By the monotonicity of $t\mapsto \partial_du(x',t)$, there are $\eta_+(x')\ge\eta_-(x')$ in $(-\ell,\ell)$ such that: $\partial_du(x',t)>0$ for $t\in(\eta_+(x'),\ell)$; $\partial_du(x',t)<0$ for $t\in(-\ell,\eta_-(x'))$; $\partial_du(x',t)=0$ for $t\in[\eta_-(x'),\eta_+(x')]$. Now, if $\eta_+(x')>\eta_-(x')$, then $\eta_\pm(x')$ are points in the regular part of the free boundary $\partial\Omega$. Thus, $\eta_\pm$ are $C^\infty$ smooth functions in a neighborhood of $x'$. If $\eta_+(x')=\eta_-(x')$, then $\Omega$ has Lebesgue density $1$ at $(x',\eta_\pm(x'))$. Finally, the $C^{1}$ regularity of $\eta_\pm$ follows from the $C^{1}$ regularity of $\partial_du$ on $\overline{\{\pm\partial_du>0\}}$. The last claim follows since thanks to the regularity of the obstacle problem $\eta_\pm$ are $C^{1,\theta}$ regular for some logarithmic modulus of continuity $\theta$, so the conclusion is a consequence of the classical boundary regularity for harmonic functions.
\end{proof}

\begin{remark}\label{rmk:cash} Notice that \cite{CaSh} Section 3, (f) case 1, shows that, in dimension $2$, $\overline{\{\pm\partial_du>0\}}$ are locally, at non isolated singular points of the free boundary of $u$, $C^1$ graphs. Thus, in view of the last sentence in the previous Lemma, the $C^2$ regularity assumption in $\overline{\Omega}$ of Theorem \ref{t:obst} is justified. 
\end{remark}

\subsection{The free boundary condition on $\text{\rm Reg}(\partial\Omega)$} In this subsection we show that the partial derivatives of the solutions to the obstacle problem solve a Bernoulli one-phase problem with an anisotropic  condition on the free boundary (for more on this class of problems see  \cite{BevilacquaCarducci}).

\begin{lemma}\label{lemma:obstacle-boundary-condition}
Let $u$ be a solution to the obstacle problem $\Delta u=\chi_{\{u>0\}}$ in an open set $D\subset\R^d$, $d\ge 2$. Suppose that $\Gamma\subset\partial\{u>0\}$ is a smooth surface contained in the regular part of the free boundary and that $\nu$ is the normal to $\Gamma$ pointing inwards $\Omega:=\{u>0\}$. Suppose that $e\in\mathbb S^{d-1}$ is a vector with $e\cdot\nu>0$ on $\Gamma$. Then, the derivative $u_e:=e\cdot\nabla u$ is $C^1$ regular up to the boundary $\Gamma$ and 
$$|\nabla u_e|^2=e\cdot\nabla u_e\quad\text{on}\quad\Gamma.$$ 
\end{lemma}
\begin{proof}
Let $0\in\Gamma\cap D$ be a regular point of $\partial\Omega$. Then, the blow-up $u_0$ of $u$ in $0$ has the form
$$u_0(x):=\frac12(x\cdot \nu)_+^2,$$
where $\nu$ is the normal to $\Gamma$ pointing inwards $\Omega$. Since $u\in C^\infty(\Gamma\cup\Omega)$ in a neighborhood of $0$ and since $\nabla^2u(0)=\nu\otimes\nu$, we can compute 
$$\nabla u_e\cdot\nu=\nabla (\nabla u\cdot e)\cdot \nu=e\cdot\nabla^2u[\nu]=e\cdot(\nu\otimes\nu)[\nu]=e\cdot\nu,$$
which shows that $\nabla u_e(0)=(e\cdot\nu)\nu$. Similarly
$$\nabla u_e\cdot e=\nabla (\nabla u\cdot e)\cdot e=e\cdot\nabla^2u[e]=e\cdot(\nu\otimes\nu)[e]=(e\cdot\nu)^2,$$
which gives precisely $|\nabla u_e|^2=e\cdot\nabla u_e.$
\end{proof}

\subsection{Proof of \cref{t:obst}}
Let $0\in \mathcal B(u)$. Then, since the unique blow up at any branching point, up to a rotation, is given by $\frac12y^2$, by our regularity assumption we can assume that $u_{yy}(0)=1$ and $u_{xy}(0)=0=u_{xx}(0)$. Thanks to Lemma \ref{l:obstacle-smooth-parametrization}, $\partial\{\pm u_y>0\}$ are $C^{1}$ regular graphs over $\{y=0\}$. Let $T$ and $S$ be the maps from Lemma \ref{l:obstacle-S-and-T}. Consider the maps 
$$T_\pm:\overline{\{\pm u_y>0\}}\to\R^2\ ,\quad T_\pm(x,y)=T(x,y)=(x-u_x,u_y),$$
which are $C^{1}$ regular and invertible in a neighborhood of $0$, and let $T_\pm^{-1}:\overline{B_r^\pm}\to\R^2$ be their inverse maps. Consider the holomorphic maps $S(T_\pm^{-1})$ and let $w^\pm:=\text{Re}(S(T_\pm^{-1}))$. Observe that, as composition of holomorphic functions, $\Delta w^\pm=0$ in $B_r^\pm$.

We will show that $w^\pm$ solve a two-membrane problem in $B_r$ and that the branching points of $\Omega$ correspond precisely to the points on the boundary of the set $\{(s,0)\ :\ w^+(s,0)<w^-(s,0)\}$. 

Notice that $T_\pm(x,y)=(x-u_x,0)$ for every $(x,y)\in \partial\{\pm u_y>0\}$ and set 
$$\Gamma_o^\pm:=\partial\{ \pm u_y>0\}\setminus \partial\{\mp u_y>0\}\qquad\text{and}\qquad \Gamma_t:=\partial\{u_y>0\}\cap \partial\{- u_y>0\}.$$ 
Given points $(x,y_+)$ and $(x,y_-)$ on $\partial\{u_y\neq0\}$, from Lemma \ref{l:obstacle-smooth-parametrization} if follows that 
\begin{itemize}
\item either $(x,y^\pm)\in \Gamma^\pm_o$, and then $T_\pm(x,y^\pm)=(x,0)$ and $y^+>y^-$; \item or $(x,y^\pm)\in \Gamma_t$, in which case $y^+=y^-$ and so $T_+(x,y)=T_-(x,y)=(x-u_x(x,y),0)$.
\end{itemize}
Since $\text{Re}(S(x,y))=-y$ on $\partial \{\pm u_y>0\}$, we have $w^+\leq w^-$ on $B_r\cap \{(s,t)\,:\,t=0\}$, with strict inequality on $T(\Gamma^\pm_o)$.
Finally, in order to deduce the free boundary condition on $T(\Gamma^\pm_o)$, we differentiate the identity
\[
w^\pm(T_\pm(x,y))=\text{Re}(S_\pm(x,y))=u_y(x,y)-y
\]
in $x$ and in $y$ obtaining 
$$\begin{cases}
\nabla w^\pm(T(x,y)) \cdot (1-u_{xx},u_{yx})=u_{xy},\\
\nabla w^\pm(T(x,y)) \cdot (-u_{xy},u_{yy})=u_{yy}-1.
\end{cases}
$$
Solving the system gives 
\begin{align}\label{eq:chepalle1}
\nabla  w^\pm (T(x,y)) &= \frac{1}{u_{yy}^2 + u_{xy}^2} 
    \begin{pmatrix}
        u_{yy} & -u_{xy} \\
        u_{xy} & u_{yy}
    \end{pmatrix}\begin{pmatrix}u_{xy}\\
 -u_{xx}\end{pmatrix}=
 \frac{1}{|\nabla u_y|^2} 
   \begin{pmatrix}u_{xy}\\
 -u_{yy}\end{pmatrix}+\begin{pmatrix}0\\
 1\end{pmatrix}.
\end{align}
Now, choosing $e=e_2$ in Lemma \ref{lemma:obstacle-boundary-condition} and using our regularity assumption we obtain that 
$$|\nabla u_y|^2=u_{yy}\quad\text{on}\quad \Gamma_o^\pm\,,$$
which combined with \eqref{eq:chepalle1} gives $\partial_t w^\pm=0$ on $T(\Gamma_o^\pm)$. In conclusion, $w^\pm:B_r^\pm\to\R$ solves the following linear thin two-membrane problem
\[
\begin{cases}
    \Delta w^\pm=0& \text{in }\{\pm t>0\}\\
    w^+\leq w^-& \text{on }\{t=0\}\\
    (w^+-w^-)\partial_tw^\pm=0& \text{on }\{t=0\}\,.
\end{cases}
\]
Standard frequency arguments on the linear thin-two membrane problem, combined with the $C^{1}$ regularity of the traces of $w^\pm$ on $\{t=0\}$, yield the desired conclusion.
\qed

\bibliographystyle{plain}
\bibliography{FreeBoundary_bib}

\end{document}